\documentclass{article}
\usepackage{amsmath,amssymb}
\usepackage[all]{xy}
\usepackage{amsthm}
\usepackage{enumerate}
\usepackage{CJK,CJKnumb}
\usepackage{color}              % ײ§³ײ²ֺֹ«
\usepackage{indentfirst}        %ֺ׳׀׀ֻץ½ר÷ך°
\usepackage{latexsym,bm}        % ´¦ְםֺ‎ׁ§¹«ֺ½ײ׀÷ֽ÷׀±ּוµִ÷ך°
\usepackage{amsmath,amssymb}    % AMSLaTeX÷ך° ׃ְֳ´ֵֵ³צ¸¼׃ֶ¯ֱֱµִ¹«ֺ½
\usepackage{graphicx}
\usepackage{cases}
\usepackage{amsthm}
\usepackage{fancyhdr}
\usepackage{graphics}
\usepackage{enumitem}
\usepackage{color}
\usepackage[all]{xy}
\usepackage{amscd}

\usepackage[T1]{fontenc}
\usepackage[utf8]{inputenc}
\usepackage{authblk}
\begin{document}

\title{Moduli spaces of arrangements of $12$ projective lines with a sextic point }
%\author[a]{Meirav Amram}
%\author[b]{Eran Lieberman}
%\author[c]{Sheng-Li Tan }
%\author[b]{Mina Teicher}
%\author[b]{Xiao-Hang Wu \thanks{Corresponding author: wuxiaoh001@foxmail.com}}
%\affil[a]{Shamoon College of Engineering, Ashdod, Israel}
%\affil[b]{Emmy Noether Research Institute for Mathematics, Bar-Ilan University, Ramat-Gan, Israel}
%\affil[c]{School of Mathematical Sciences, East China Normal University, Shanghai, P.~R.~China}

\author{Meirav Amram,~Eran Lieberman,~Sheng-Li Tan,~Mina Teicher,~Xiao-Hang Wu$^*$}

\date{A paper from 2018}
  \maketitle
  \begin{abstract}
In this paper, we try to classify moduli spaces of arrangements of $12$ lines with sextic points. We show
that moduli spaces of arrangements of $12$ lines with sextic points can consist of more than two connected components. We also present defining equations of the arrangements whose moduli spaces are not irreducible taking quotients by the complex conjugation by supply some potential Zariski pairs.Through complex conjugation we take quotients and supply some potential Zariski pairs.
\end{abstract}

\footnote{\textbf{The paper was written along the year 2018, during the postdoctoral position of the 2nd and 5th co-authors.}}

  \section{Introduction}

  \newtheorem{Theorem}{Theorem}[section]
  \newtheorem{Corollary}{Corollary}[section]
  \newtheorem{Proposition}{Proposition}[section]
  \newtheorem{Lemma}{Lemma}[section]
   \newtheorem{Question}{Question}[section]
    \newtheorem{Example}{Example}[section]
    \newtheorem{Definition}{Definition}[section]

 A line arrangement $\mathcal{A}$  in $\mathbb{CP}^2$ is a finite collection of projective lines. The complement of the union of lines in A is denoted as $M(\mathcal{A})$. We call the set $L(\mathcal{A}) = \{\cap_{i\in S}
\mid  S\in\{1, 2,...,n\}\}$ partially ordered
by reverse inclusion the intersection lattice of $\mathcal{A}$.

 An essential topic in hyperplane arrangement theory is to study the intersection between topology of complements and combinatorics of intersection lattices. It is important to study how closely topology and combinatorics of a given arrangement are related. For line arrangements, Jiang and Yau~\cite{Jia2}~ showed that homeomorphism of the complement
 always implies lattice isomorphism. However, the converse is not true in general for line arrangements. In~\cite{Jia}~and~\cite{Wan}~, the authors found a large class of line arrangements whose intersection lattices determine topology of the complements, called nice arrangements of hyperplanes in higher dimensional projective
 spaces(see~\cite{Wan2,Wan3,Yau}~).

 We call a pair of line arrangements a $Zariski\ pair$ if they are lattice isomorphic, but the fundamental groups of their complements are different. The first $Zariski\ pair$ of line arrangements was constructed by Rybnikov~\cite{Ryb}~. On
 the other hand, combining the results of Fan~\cite{Fan}~, Garber et al.~\cite{Gar}~ proved that there is no Zariski pair of arrangements of up to $8$ real lines. This result was recently generalized to arrangements of $8$ complex lines by Nazir and Yoshinaga~\cite{Naz}~. In the same paper, Nazir and Yoshinaga also claimed that there is no Zariski pair of arrangements of $9$ complex line. A complete proof of their claim was presented in~\cite{Ye}~. Recently, Amram et al. classified arrangements of $10$ complex lines and $11$ complex lines with a quintuple point in ~\cite{Amr,Amr1,Amr2}~ and found some "potential Zariski pairs".

 Let $\mathcal{A}$ be a complex line arrangement. We define the $moduli space$ of  line arrangements with the fixed lattice $L(\mathcal{A})$ (or simply, the moduli space of $\mathcal{A}$)as
 $$\mathcal{M}_\mathcal{A}=\{\mathcal{B}\in((\mathbb{CP}^2)^*)^n\mid \mathcal{B}\sim\mathcal{A}/PGL_\mathbb{C}(2)\}$$
 where $\mathcal{B}\sim\mathcal{A}$ means $\mathcal{B}$ and $\mathcal{A}$ are lattice isomorphic. We denote by $\mathcal{M}^c_\mathcal{A}$ the quotient of $\mathcal{M}_\mathcal{A}$ under the complex conjugation. By Randell's lattice-isotopy theorem in ~\cite{Ran}~and Cohen and Suciu's theorem~\cite{Coh}~, we know that arrangements in the same connected component of the moduli space, or in two complex conjugate components, can not form Zariski pairs.

 The classification of moduli spaces consists of three steps. First, we will roughly classify intersection
lattices according the number of multiple intersection points. Second, we divide our classification into some
different cases according to different positions between quintuple point and the other multiple intersection
points. Third, we will write down defining equations involving parameters for a given intersection lattice.

This paper is structured as follows. Section 2 provides preliminaries and ideas on classifying moduli
spaces of arrangements of $12$ lines. Section 3 shows that moduli spaces of arrangements with multiple points of
high multiplicity are most likely irreducible. In Section 4 , we completely classify the arrangements
of $12$ lines with a sextic point and at least one quadruple point. In Section 5 we deal with the arrangements
of $12$ lines with a sextic point and no quadruple point.

\section*{Acknowledgement}
This research was supported by the ISF-NSFC joint research program of the 1st and 3rd co-authors (grant No. 2452/17). The paper was written in 2018, along the postdoctoral position of the 2nd and 5th co-authors, under a common supervision of the 1st and 4th co-authors. Both co-authors, the 2nd and the 5th ones, were financially supported by the ISF research grant of the 1st author.

\section{Preliminaries}

  Let $A=\{L_1, L_2,\cdots, L_n\}$ be a line arrangement in $\mathbb{CP}^2$. We say a singularity of
  $L_1 \cup L_2\cup\cdots\cup L_n $ is a multiple point of $\mathcal{A}$ if it has multiplicity of at least $3$. We call the set
   $L(\mathcal{A}) = \{\bigcap_{i\in S}L_i\mid S \subseteq\{1, 2, \ldots, n\}\}$ partially ordered by reverse inclusion in the intersection lattice of $\mathcal{A}$.

   \begin{Definition}~
Two line arrangements $\mathcal{A}_1$ and $\mathcal{A}_2$ are lattice isomorphic, denoted as $\mathcal{A}_1\sim\mathcal{A}_2$ , if their intersection lattices $L(\mathcal{A}_1)$ and $L(\mathcal{A}_2)$ are isomorphic, i.e. there is a permutation $\phi$ of ${1, 2,\ldots, n}$ such that
$$dim(\bigcap\limits_{i\in S,L_i\in\mathcal{A}_1}L_i)=dim(\bigcap\limits_{j\in \phi(S),H_j\in\mathcal{A}_2}H_j)$$
   for any nonempty subset $S\subseteq\{1,2,\ldots,n\}$.
   \end{Definition}

  \begin{Definition}~\cite{Ryb}
Let $k\in N$. We say that a line arrangement $\mathcal{A}$ is of type $C_k$ if $k$ is the
minimum number of lines in $\mathcal{A}$ containing all points of multiplicity of at least three.
\end{Definition}
\begin{Definition}~\cite{Ryb}
Let $\mathcal{A}$ be an line arrangement of type $C_3$ . Then $\mathcal{A}$ is a simple $C_3$ arrangement if there are three lines $L_1, L_2, L_3 \in \mathcal{A}$ such that all points of multiplicity of at least three are
contained in $L_1\cup L_2\cup L_3$ and one of the following holds:
\begin{enumerate}
    \item $L_1\cap L_2\cap L_3\neq\varnothing$, or
    \item $L_1\cap L_2\cap L_3=\varnothing$ and one of $L_1$ , $L_2$ , and $L_3$ contains only one multiple point apart from the other two lines.

    \end{enumerate}
\end{Definition}

\begin{Theorem}~\cite{Ryb}
Let $\mathcal{A}$ be an arrangement of $C_3$ of simple type. Then the moduli space $\mathcal{M}_{\mathcal{A}}$ is irreducible.

\end{Theorem}

\begin{Theorem}~\cite{Ryb}
Let $\mathcal{A} = \{L_1, L_2,\cdots, L_n\}$ be a line arrangement. Assume that Ln passes
through at most $2$ multiple points. Set $\mathcal{A}'= \{L_1, L_2, \cdots , L_{n-1}\}$, and then $\mathcal{M}_{\mathcal{A}}$ is irreducible if $\mathcal{M}_{\mathcal{A}'}$ is
irreducible.

\end{Theorem}

We say that a line arrangement is nonreductive if each line of the arrangement passes through at least $3$
multiple points. Otherwise, we say the arrangement is reductive.
Denote by $n_r$ the number of intersection points of multiplicity $r$ . We recall the following useful results.

\begin{Lemma}\label{2}~\cite{Hir}
Let $\mathcal{A}$ be an arrangement of $k$ lines in $\mathbb{CP}^2$. Then
$$\frac{k(k-1)}{2}=\sum\limits_{r\geq2}\frac{r(r-1)n_r}{2}$$

\end{Lemma}
\begin{Theorem}\label{3}~\cite{Hir}
Let $\mathcal{A}$ be an arrangement of $k$ lines in  $\mathbb{CP}^2$.
Assume that $n_k = n_{k-1} = n_{k-2} = 0$.
Then$$n_2+\frac{3}{4}n_3\geq k+\sum\limits_{r\geq5}(2r-9)n_r.$$
\end{Theorem}

\begin{Theorem}\label{1}\cite{Amr}
Let $\mathcal{A}$ be an arrangement of $n (n \geq 9)$ lines. If there is a multiple point of multiplicity $\geq n-4$,
then the moduli space $\mathcal{M}_\mathcal{A}$ is irreducible.
\end{Theorem}

The following lemma is well known and is used to facilitate the calculation in our paper.

\begin{Lemma}\label{jisuan}
 Let $\{L_1, L_2, L_3\}$ and $\{L_4, L_5, L_6, L_7, L_8\}$ be two pencils of lines who intersect at one point and
intersect transversally in 15 points. Then there is an automorphism of the dual projective plane such that the
8 lines under the automorphism are defined by $Y = 0,\ Y=Z,\ Y = t_1Z,\  X = 0,\ X=Z,\ X = t_2Z,\ X = t_3Z,\ X = t_4Z$.
\end{Lemma}
\section{Arrangements of $12$ lines with multiple points of multiplicity at least $7$}
\begin{Theorem}
Let $\mathcal{A}$ be an arrangement of $12$ lines with a multiple point of multiplicity $\geq7$, $\mathcal{A}$ is not reductive. Then the moduli space $\mathcal{M}_\mathcal{A}$ is irreducible.
\end{Theorem}
\begin{proof}
Let $\mathcal{A}$ be an arrangement of $12$ lines with a multiple point of multiplicity $\geq8$. Since the theorem \ref{1}, the moduli space $\mathcal{M}_\mathcal{A}$ is irreducible. Let $\mathcal{A}$ be an arrangement of $12$ lines with a multiple point of multiplicity $7$ and no multiple points of higher multiplicities. By Lemma \ref{2} and Theorem \ref{3} we have $$54-26n_7\geq\frac{9}{4}(n_3+n_4+n_5+n_6)$$
On the other hand, it is easy to see that there must be at least $15-n_7$ multiple points of multiplicity $<7$.
Thus, $15-n_7 \leq n_3 + n_4 + n_5$.  Then we get $n_7 \leq\frac{54}{95}<1$, a contradiction.

\end{proof}

\section{Arrangements of 12 lines with a sextic point and one quadruple points}

\begin{Theorem}\label{4}
Let $\mathcal{A}$ be a nonreductive arrangement of $12$ lines in $\mathbb{CP}^2$ with a sextic point and $n_r = 0$ for
$r \geq 7$. Then $n_6 = 1, n_5=0$ and $n_4 \leq1$.

\end{Theorem}
\begin{proof}
By Lemma \ref{2} and Theorem \ref{3} we have $$54-18n_6\geq\frac{9}{4}(n_3+n_4+n_5).$$
On the other hand, it is easy to see that there must be at least $13-n_6$ multiple points of multiplicity $<6$. Thus,
$$\frac{9}{4}(13-n_6)\leq54-18n_6.$$
It follows that $n_6\leq1$, then $n_6=1$.

 By Lemma \ref{2} and Theorem \ref{3} we have $$36-11n_5\geq\frac{9}{4}(n_3+n_4).$$
On the other hand, it is easy to see that there must be at least $12-n_5$ multiple points of multiplicity $<5$. Thus,
$$\frac{9}{4}(12-n_5)\leq36-11n_5.$$ It follows that $n_5\leq1$. If $n_5=1$ and the quintuple point and the sextic point are not collinear,  then it is easy to see
that there is a line with at most 2 multiple points. If $n_5=1$ and the quintuple point and the sextic point are collinear.  Let
$L_1\cap L_2\cap L_3\cap L_4\cap L_5\cap L_{12}$ and $L_6\cap L_7\cap L_8\cap L_9\cap L_{12}$ be the 2 points.  and let $L_{12}$ be the line at infinity. Each of $L_{10}$ and $L_{11}$ must pass through $4$ points of $L_i\cap L_j$ , $i = 1, 2, 3, 4, 5; j = 6, 7, 8, 9$. Then one of $L_i, i=1,2,3,4,5$  at most 2 multiple points, a contradiction. Then $n_5=0$.

By Lemma \ref{2} and Theorem \ref{3} we have $$36-6n_4\geq\frac{9}{4}n_3.$$
On the other hand, it is easy to see that there must be at least $12-n_4$ multiple points of multiplicity 3. It follows that $n_4\leq2$. First, if  a sextic point and a quadruple point  are not collinear, then it is easy to see that there must be at least $12$ multiple points of multiplicity 3, it follows that $n_4\leq1$. Second, if $n_4=2$, and a sextic point and a quadruple point  are  collinear. Let
$L_1\cap L_2\cap L_3\cap L_4\cap L_5\cap L_{12}$ be a sextic point and $L_6\cap L_7\cap L_8\cap L_{12}$ be a quadruple point, and let $L_{12}$ be the line at infinity. Sice the arrangement is nonreductive, then the $L_{12}$ has the at least $3$ multiple points. If there is another quadruple point in the $L_{12}$. So $L_9\cap L_{10}\cap L_{11}\cap L_{12}\neq\varnothing$. Then Each of $L_{10}$ , $L_{11}$ and $L_{9}$ must pass through most of $4$ points of $L_i\cap L_j$ , $i = 1, 2, 3, 4, 5; j = 6, 7, 8$. Then one of $L_i, i=1,2,3,4,5$  at most 2 multiple points, a contradiction. If there is not another quintuple point in the $L_{12}$. We can assume $L_9\cap L_{10}\cap L_{12}\neq\varnothing$, $L_9\cap L_{10}\cap L_{1}\cap L_8\neq\varnothing$ and $L_{11}\cap L_{1}\cap L_7\neq\varnothing$. Then Each of $L_{10}$ and $L_{9}$ must pass through most of $2$ points of $L_i\cap L_j$ , $i = 2, 3, 4, 5; j = 6, 7$. $L_{11}$ must pass through most of $2$ points of $L_i\cap L_j$ , $i = 2, 3, 4, 5; j = 6, 8$. Then one of $L_i, i=1,2,3,4,5$  at most 2 multiple points, a contradiction. Then we get $n_4\leq1.$

\end{proof}
\begin{Theorem}
Let $\mathcal{A}$ be a nonreductive arrangement of $12$ lines in $\mathbb{CP}^2$ with $n_6 = n_4 = 1$ and $n_r = 0$ for
$r \geq 7$. If the sextic point and the quadruple point are not collinear, then $\mathcal{M}_\mathcal{A}$ is empty.
\end{Theorem}
\begin{proof}
Let
$L_1\cap L_2\cap L_3\cap L_4\cap L_5\cap L_{9}$ be a sextic point and $L_6\cap L_7\cap L_8\cap L_{12}$ be a quadruple point, and let $L_{12}$ be the line at infinity. Then Each of $L_{10}$, $L_{11}$  must pass through most of $2$ points of $L_i\cap L_j$ , $i =1,  2, 3, 4, 5, 9; j = 6, 7, 8, 12$. Then one of $L_i, i=1,2,3,4,5, 9$  at most 2 multiple points, a contradiction.
\end{proof}

\begin{Theorem}
Let $\mathcal{A}$ be a nonreductive arrangement of $12$ lines in $\mathbb{CP}^2$ with $n_6 = n_4 = 1$ and $n_r = 0$ for
$r \geq 7$. If the sextic point and the quadruple point are collinear, then $\mathcal{M}_\mathcal{A}$ or $\mathcal{M}^c_\mathcal{A}$ is irreducible except in
the case of Figure 12 and the corresponding
arrangement of this figure is ¡°potential Zariski pair¡±.

\end{Theorem}
\begin{proof}
Assume that $L_1 \cap L_2 \cap L_3 \cap L_{12}$ is the quadruple point and $L_4 \cap L_5 \cap L_6 \cap L_7 \cap L_8 \cap L_{12}$ is the sextic point. Then one of ${L_{10} \cap L_9, L_{11} \cap L_{10}, L_9 \cap L_{11}}$ is on $L_{12}$ so that it contains at least 3 multiple points. We may assume $L_{10} \cap L_{11}$ is on $L_{12}$. Because $\{L_{9}\cup L_{10}\cup L_{11}\}\cap \{L_i\cap L_j:1\leq i\leq3,4\leq j\leq8\}$ are at most $9$ points. Then  at least one multiple  point is not on
$\{L_i\cap L_j:1\leq i\leq3,4\leq j\leq8\}$ but is on $L_4 \cup L_5 \cup L_6 \cup L_7 \cup L_8$. It's easy to see that  at most two multiple points are not on
$\{L_i\cap L_j:4\leq i<j\leq8\}$ but is on $L_4 \cup L_5 \cup L_6 \cup L_7 \cup L_8$.

\begin{itemize}
\item[Case 1] Only one multiple  point is not on
$\{L_i\cap L_j:1\leq i\leq3,4\leq j\leq8\}$ but is on $L_4 \cup L_5 \cup L_6 \cup L_7 \cup L_8$. We can assume the point is $L_{9}\cap L_{10}\cap L_{4}$. Up to a lattice isomorphism, we assume $L_{9}$ passes through $\{L_3 \cap L_5, L_2 \cap L_6, L_1 \cap L_7\}$ and $L_{3}\cap L_{4}$ is on $L_{11}$.
\begin{enumerate}
\item[\rm 1)] $L_{11}$ passes through $L_2 \cap L_5$. Then $L_{11}$ passes through $L_1 \cap L_6$ or $L_1 \cap L_8$. If $L_{11}$ passes through $L_1 \cap L_6$. Since $L_8$ passes through at most 2 multiple points,a contradiction. Then $L_{11}$ passes through $L_1 \cap L_8$.
    \begin{enumerate}
    \item[\rm I]$L_{10}$ passes through $\{L_1 \cap L_6, L_3 \cap L_7, L_2 \cap L_8\}$ (Figure 1).

    By Lemma \ref{jisuan}, we can let $L_1=\{Y = 0\},\ L_2=\{Y=Z\},\ L_3=\{Y = t_1Z\},\ L_4=\{ X = 0\},\ L_5=\{X=Z\},\ L_6=\{X = t_2Z\},\ L_7=\{X = t_3Z\},\ L_8=\{X = t_4Z\}$. Then $L_{11}$ pass through $(0,t_1,1)$,$(1,1,1)$,$(t_4,0,1)$, $L_{9}$ pass through $(1,t_1,1)$,$(t_2,1,1)$,\\$(t_3,0,1)$, $L_{10}$ pass through $(t_3,t_1,1)$,$(t_4,1,1)$,$(t_2,0,1)$. Since $L_{12}$ pass through $L_{10} \cap L_{11}$, then $\frac{t_1}{t_4}=\frac{t_1-1}{t_4-t_3}$. Since $L_{4}$ pass through $L_{10} \cap L_{9}$,then $\frac{t_1t_3}{t_3-1}=1-\frac{(t_1-1)t_4}{t_3-t_4}$.After an easy computation, we conclude that Figures 1 cannot be realized.

\item[\rm II] $L_{10}$ passes through $\{L_1 \cap L_6, L_2 \cap L_7, L_3 \cap L_8\}$ (Figure 2).
\end{enumerate}
   \begin{figure}[h]
\begin{minipage}[t]{0.5\linewidth}
\centering
\includegraphics[width=6.5cm,height=4.5cm]{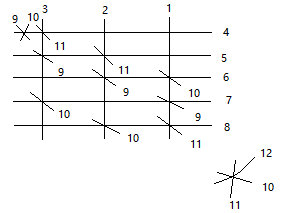}
\caption{}
\end{minipage}
\begin{minipage}[t]{0.5\linewidth}        %ֽ¼ֶ¬ױ¼׃ֳׂ»׀׀¿ם¶ָµִ45%
\hspace{2mm}
\includegraphics[width=6.5cm,height=4.5cm]{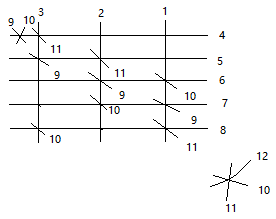}
\caption{}
\end{minipage}
\end{figure}

  After an easy computation(by Lemma \ref{jisuan}), we conclude that Figures  2 cannot be realized.
\item[\rm 2)] $L_{11}$ passes through $L_1 \cap L_5$. Then $L_{11}$ passes through $L_2 \cap L_8$ or $L_2 \cap L_7$. If $L_{11}$ passes through $L_2 \cap L_7$. Sice $L_8$ passes through at most 2 multiple points, then $L_{11}$ passes through $L_2 \cap L_8$.
    \begin{enumerate}
    \item[\rm I]$L_{10}$ passes through $\{L_3 \cap L_6, L_2 \cap L_7, L_1 \cap L_8\}$ (Figure 3).

\item[\rm II] $L_{10}$ passes through $\{L_1 \cap L_6, L_2 \cap L_7, L_3 \cap L_8\}$ (Figure 4).

 \end{enumerate}
     \begin{figure}[h]
\begin{minipage}[t]{0.5\linewidth}
\centering
\includegraphics[width=6.5cm,height=4.5cm]{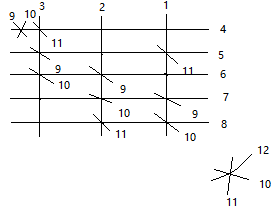}
\caption{}
\end{minipage}
\begin{minipage}[t]{0.5\linewidth}        %ֽ¼ֶ¬ױ¼׃ֳׂ»׀׀¿ם¶ָµִ45%
\hspace{2mm}
\includegraphics[width=6.5cm,height=4.5cm]{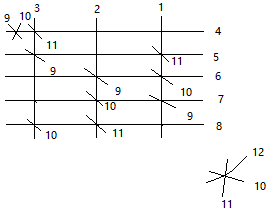}
\caption{}
\end{minipage}
\end{figure}

 After an easy computation(by Lemma \ref{jisuan}), we conclude that Figures 3 and 4 cannot be realized.
\item[\rm 3)] $L_{10}$ passes through $\{L_2 \cap L_5, L_3 \cap L_7, L_2 \cap L_8\}$ ,$L_{11}$ passes through $\{L_1 \cap L_6, L_2 \cap L_8\}$(Figure 5).
\item[\rm 4)] $L_{10}$ passes through $\{L_1 \cap L_6, L_3 \cap L_7, L_1 \cap L_8\}$ . $L_{11}$ passes through $\{L_2 \cap L_5, L_1 \cap L_8\}$(Figure 6).
    \begin{figure}[h]
\begin{minipage}[t]{0.5\linewidth}
\centering
\includegraphics[width=6.5cm,height=4.5cm]{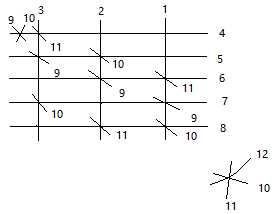}
\caption{}
\end{minipage}
\begin{minipage}[t]{0.5\linewidth}        %ֽ¼ֶ¬ױ¼׃ֳׂ»׀׀¿ם¶ָµִ45%
\hspace{2mm}
\includegraphics[width=6.5cm,height=4.5cm]{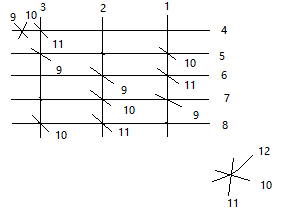}
\caption{}
\end{minipage}
\end{figure}

 After an easy computation(by Lemma \ref{jisuan}), we conclude that Figures 5 and 6 cannot be realized.
 \end{enumerate}
\item[Case 2] two multiple  point is not on
$\{L_i\cap L_j:4\leq i<j\leq8\}$ but is on $L_4 \cup L_5 \cup L_6 \cup L_7 \cup L_8$. We can assume the two points are $L_{9}\cap L_{10}\cap L_{4}$ and $L_{9}\cap L_{11}\cap L_{5}$.
\begin{itemize}
\item[Subcase 1] $L_9$ pass through 3 points of $\{L_i\cap L_j:1\leq i\leq3,4\leq j\leq8\}$. $L_{10}$ pass through 2 points of $\{L_i\cap L_j:1\leq i\leq3,4\leq j\leq8\}$. $L_{11}$ pass through 3 points of $\{L_i\cap L_j:1\leq i\leq3,4\leq j\leq8\}$. Up to a lattice isomorphism, we assume $L_{9}$ passes through $\{L_3 \cap L_6, L_2 \cap L_7, L_1 \cap L_8\}$, $L_{3}\cap L_{4}$ is on $L_{11}$.
    \begin{enumerate}
    \item[\rm 1)] $L_{10}$ passes through $L_{3}\cap L_{5}$

        \begin{enumerate}
        \item[\rm I)]$L_{10}$ passes through $ L_2 \cap L_8$. $L_{11}$ passes through $\{L_1 \cap L_7, L_2 \cap L_6\}$(Figure 7).

\item[\rm II)] $L_{10}$ passes through $ L_1 \cap L_7$. $L_{11}$ passes through $\{L_1 \cap L_6, L_2 \cap L_8\}$(Figure 8).
    \begin{figure}[h]
\begin{minipage}[t]{0.5\linewidth}
\centering
\includegraphics[width=6.5cm,height=4.5cm]{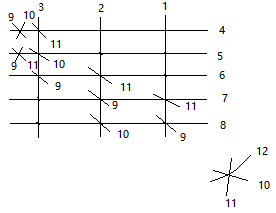}
\caption{}
\end{minipage}
\begin{minipage}[t]{0.5\linewidth}        %ֽ¼ֶ¬ױ¼׃ֳׂ»׀׀¿ם¶ָµִ45%
\hspace{2mm}
\includegraphics[width=6.5cm,height=4.5cm]{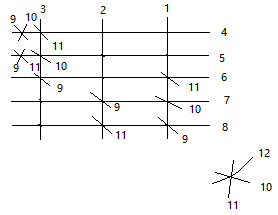}
\caption{}
\end{minipage}
\end{figure}

After an easy computation(by Lemma \ref{jisuan}), we conclude that Figures 7 and 8 cannot be realized
  \item[\rm III] $L_{10}$ passes through $L_1 \cap L_6$. $L_{11}$ passes through $\{L_1 \cap L_7, L_2 \cap L_8\}$(Figure 9).
\item[\rm IV] $L_{10}$ passes through $ L_2 \cap L_7$.  $L_{11}$ passes through $\{L_1 \cap L_7, L_2 \cap L_8\}$(Figure 10).
    \end{enumerate}
    \begin{figure}[h]
\begin{minipage}[t]{0.5\linewidth}
\centering
\includegraphics[width=6.5cm,height=4.5cm]{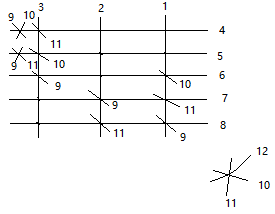}
\caption{}
\end{minipage}
\begin{minipage}[t]{0.5\linewidth}        %ֽ¼ֶ¬ױ¼׃ֳׂ»׀׀¿ם¶ָµִ45%
\hspace{2mm}
\includegraphics[width=6.5cm,height=4.5cm]{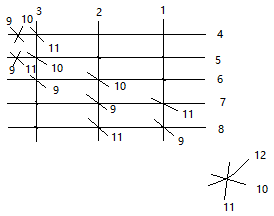}
\caption{}
\end{minipage}
\end{figure}

\item[\rm 2)] $L_{10}$ passes through $L_2\cap L_5$.
    \begin{enumerate}
    \item[\rm I]$L_{10}$ passes through $ L_3 \cap L_8$.  $L_{11}$ passes through $\{L_1 \cap L_7, L_2 \cap L_6\}$(Figure 11).

    After an easy computation(by Lemma \ref{jisuan}), we conclude that Figures 9, 10 and 11  cannot be realized.
\item[\rm II] $L_{11}$ passes through $\{L_1 \cap L_6, L_2 \cap L_8\}$. $L_{10}$ passes through $\{L_2\cap L_5, L_3 \cap L_7\}$ (Figure 12).
    \begin{figure}[h]
\begin{minipage}[t]{0.5\linewidth}
\centering
\includegraphics[width=6.5cm,height=4.5cm]{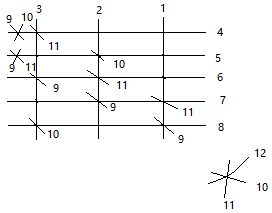}
\caption{}
\end{minipage}
\begin{minipage}[t]{0.5\linewidth}        %ֽ¼ֶ¬ױ¼׃ֳׂ»׀׀¿ם¶ָµִ45%
\hspace{2mm}
\includegraphics[width=6.5cm,height=4.5cm]{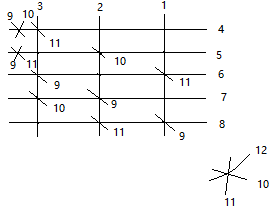}
\caption{}
\end{minipage}
\end{figure}

    Figure 12 can be defined by the following equation:
    $XYZ(Y-Z)(Y-t_1Z)(X-Z)(x-t_2Z)(x-t_3Z)(x-t_4Z)[X+(t_3-t_4)Y-t_3Z][X+(t_4-t_2)Y-t_4Z][X+(1-t_3)Y-Z]=0,$ where
    $t_1=\frac{1-t}{1-t-t^2}, t_2=t^2, t_3=t, t_4=1-t$, and satisfies $t^3+t^2+t-1=0$.

    \item[\rm III] $L_{10}$ passes through $L_1 \cap L_7$. $L_{11}$ passes through $\{L_1 \cap L_6, L_2 \cap L_8\}$(Figure 13).
\item[\rm IV] $L_{10}$ passes through $L_1 \cap L_6$. $L_{11}$ passes through $\{L_1 \cap L_7, L_2 \cap L_8\}$(Figure 14).
    \begin{figure}[h]
\begin{minipage}[t]{0.5\linewidth}
\centering
\includegraphics[width=6.5cm,height=4.5cm]{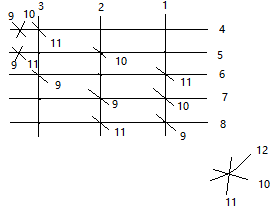}
\caption{}
\end{minipage}
\begin{minipage}[t]{0.5\linewidth}        %ֽ¼ֶ¬ױ¼׃ֳׂ»׀׀¿ם¶ָµִ45%
\hspace{2mm}
\includegraphics[width=6.5cm,height=4.5cm]{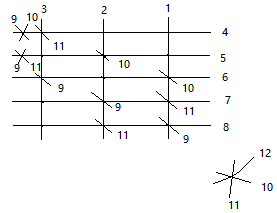}
\caption{}
\end{minipage}
\end{figure}

\end{enumerate}
After an easy computation(by Lemma \ref{jisuan}), we conclude that Figures 13 and 14 cannot be realized
\item[\rm 3)] $L_{10}$ passes through $L_1\cap L_5$. This case is lattice isomorphic to 2).
    \end{enumerate}

\item[Subcase 2]$L_9$ pass through 2 points of $\{L_i\cap L_j:1\leq i\leq3,4\leq j\leq8\}$. $L_{10}$ pass through 3 points of $\{L_i\cap L_j:1\leq i\leq3,4\leq j\leq8\}$. $L_{11}$ pass through 3 points of $\{L_i\cap L_j:1\leq i\leq3,4\leq j\leq8\}$. Up to a lattice isomorphism, we assume $L_{9}$ passes through $\{L_3 \cap L_6, L_2 \cap L_7\}$.

    \begin{enumerate}
    \item[\rm 1)]$L_{10}$ passes through $L_3\cap L_5$.
    \begin{enumerate}

    \item[\rm I] $L_{11}$ passes through $\{L_{3}\cap L_{4}, L_2 \cap L_6, L_1 \cap L_8\}$. $L_{10}$ passes through $\{L_3\cap L_5, L_1 \cap L_7, L_2 \cap L_8\}$ (Figure 15).
\item[\rm II] $L_{11}$ passes through $\{L_{3}\cap L_{4}, L_1 \cap L_7, L_2 \cap L_8\}$. $L_{10}$ passes through $\{L_3\cap L_5, L_2 \cap L_6, L_1 \cap L_8\}$ (Figure 16).
    \begin{figure}[h]
\begin{minipage}[t]{0.5\linewidth}
\centering
\includegraphics[width=6.5cm,height=4.5cm]{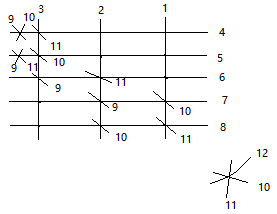}
\caption{}
\end{minipage}
\begin{minipage}[t]{0.5\linewidth}        %ֽ¼ֶ¬ױ¼׃ֳׂ»׀׀¿ם¶ָµִ45%
\hspace{2mm}
\includegraphics[width=6.5cm,height=4.5cm]{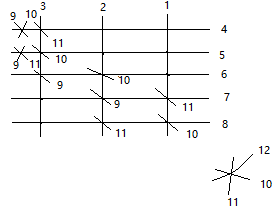}
\caption{}
\end{minipage}
\end{figure}

 After an easy computation, we conclude that Figures 15 and 16 cannot be realized.
\item[\rm III] $L_{11}$ passes through $\{L_{3}\cap L_{4}, L_2 \cap L_6, L_1 \cap L_8\}$. $L_{10}$ passes through $\{L_2\cap L_5, L_1 \cap L_7, L_3 \cap L_8\}$ (Figure 17).
\item[\rm IV] $L_{11}$ passes through $\{L_{3}\cap L_{4}, L_1 \cap L_6, L_2 \cap L_8\}$. $L_{10}$ passes through $\{L_2\cap L_5, L_1 \cap L_8, L_3 \cap L_7\}$ (Figure 18).
    \end{enumerate}
    \begin{figure}[h]
\begin{minipage}[t]{0.5\linewidth}
\centering
\includegraphics[width=6.5cm,height=4.5cm]{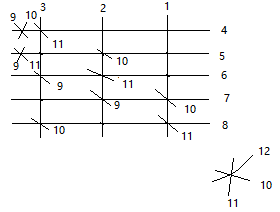}
\caption{}
\end{minipage}
\begin{minipage}[t]{0.5\linewidth}        %ֽ¼ֶ¬ױ¼׃ֳׂ»׀׀¿ם¶ָµִ45%
\hspace{2mm}
\includegraphics[width=6.5cm,height=4.5cm]{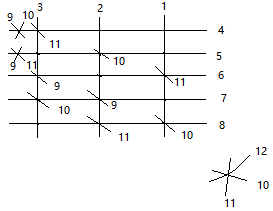}
\caption{}
\end{minipage}
\end{figure}

 After an easy computation, we conclude that Figures 17 and 18 cannot be realized.
\item[\rm 2)]$L_{10}$ passes through $L_2\cap L_5$.
\begin{enumerate}
\item[\rm I] $L_{11}$ passes through $\{L_{3}\cap L_{4}, L_1 \cap L_6, L_2 \cap L_8\}$. $L_{10}$ passes through $\{L_2\cap L_5, L_1 \cap L_7, L_3 \cap L_8\}$ (Figure 19).
\item[\rm II] $L_{11}$ passes through $\{L_{3}\cap L_{4}, L_1 \cap L_7, L_2 \cap L_8\}$. $L_{10}$ passes through $\{L_2\cap L_5, L_1 \cap L_6, L_3 \cap L_8\}$ (Figure 20).
    \begin{figure}[h]
\begin{minipage}[t]{0.5\linewidth}
\centering
\includegraphics[width=6.5cm,height=4.5cm]{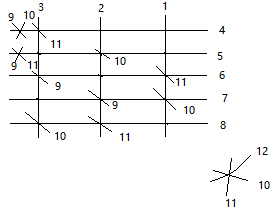}
\caption{}
\end{minipage}
\begin{minipage}[t]{0.5\linewidth}        %ֽ¼ֶ¬ױ¼׃ֳׂ»׀׀¿ם¶ָµִ45%
\hspace{2mm}
\includegraphics[width=6.5cm,height=4.5cm]{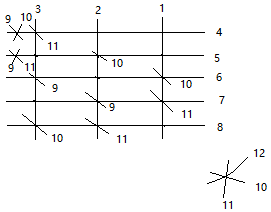}
\caption{}
\end{minipage}
\end{figure}

 After an easy computation, we conclude that Figures 19 and 20 cannot be realized.
\item[\rm III] $L_{11}$ passes through $\{L_{3}\cap L_{4}, L_2 \cap L_6, L_1 \cap L_8\}$. $L_{10}$ passes through $\{L_1\cap L_5, L_3 \cap L_7, L_2 \cap L_8\}$ (Figure 21).
\item[\rm IV] $L_{11}$ passes through $\{L_{3}\cap L_{4}, L_1 \cap L_7, L_2 \cap L_8\}$. $L_{10}$ passes through $\{L_1\cap L_5, L_2 \cap L_6, L_3 \cap L_8\}$ (Figure 22).
    \end{enumerate}
    \begin{figure}[h]
\begin{minipage}[t]{0.5\linewidth}
\centering
\includegraphics[width=6.5cm,height=4.5cm]{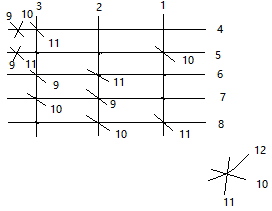}
\caption{}
\end{minipage}
\begin{minipage}[t]{0.5\linewidth}        %ֽ¼ֶ¬ױ¼׃ֳׂ»׀׀¿ם¶ָµִ45%
\hspace{2mm}
\includegraphics[width=6.5cm,height=4.5cm]{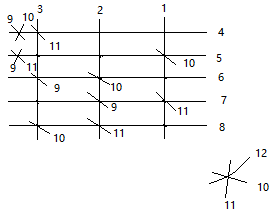}
\caption{}
\end{minipage}
\end{figure}

 After an easy computation, we conclude that Figures 21 and 22 cannot be realized.
 \item[\rm 3] $L_{10}$ passes through $L_1\cap L_5$.
 \begin{enumerate}
\item[\rm I] $L_{11}$ passes through $\{L_{1}\cap L_{4}, L_2 \cap L_6, L_3 \cap L_8\}$. $L_{10}$ passes through $\{L_1\cap L_5, L_3 \cap L_7, L_2 \cap L_8\}$ (Figure 23).
\item[\rm II] $L_{11}$ passes through $\{L_{1}\cap L_{4}, L_3 \cap L_7, L_2 \cap L_8\}$. $L_{10}$ passes through $\{L_1\cap L_5, L_2 \cap L_6, L_3 \cap L_8\}$ (Figure 24).
    \end{enumerate}
    \begin{figure}[h]
\begin{minipage}[t]{0.5\linewidth}
\centering
\includegraphics[width=6.5cm,height=4.5cm]{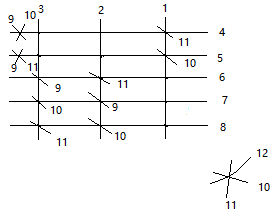}
\caption{}
\end{minipage}
\begin{minipage}[t]{0.5\linewidth}        %ֽ¼ֶ¬ױ¼׃ֳׂ»׀׀¿ם¶ָµִ45%
\hspace{2mm}
\includegraphics[width=6.5cm,height=4.5cm]{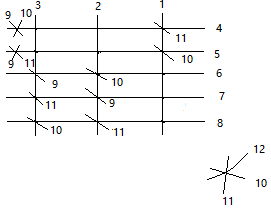}
\caption{}
\end{minipage}
\end{figure}

 After an easy computation, we conclude that Figures 23 and 24 cannot be realized.
    \end{enumerate}

\item[Subcase 3]$L_9$ pass through 3 points of $\{L_i\cap L_j:1\leq i\leq3,4\leq j\leq8\}$. $L_{10}$ pass through 3 points of $\{L_i\cap L_j:1\leq i\leq3,4\leq j\leq8\}$. $L_{11}$ pass through 3 points of $\{L_i\cap L_j:1\leq i\leq3,4\leq j\leq8\}$. Up to a lattice isomorphism, we assume $L_{9}$ passes through $\{L_3 \cap L_6, L_2 \cap L_7, L_1 \cap L_8\}$, $L_{3}\cap L_{4}$ is on $L_{11}$.
    \begin{enumerate}
    \item[\rm 1)] $L_{11}$ passes through $\{L_1 \cap L_7, L_2 \cap L_6\}$. $L_{10}$ passes through $\{L_3\cap L_5, L_1 \cap L_6, L_2 \cap L_8\}$ (Figure 25).
\item[\rm 2)] $L_{11}$ passes through $\{L_1 \cap L_6, L_2 \cap L_8\}$. $L_{10}$ passes through $\{L_2\cap L_5, L_1 \cap L_7, L_3 \cap L_8\}$ (Figure 26).
    \begin{figure}[h]
\begin{minipage}[t]{0.5\linewidth}
\centering
\includegraphics[width=6.5cm,height=4.5cm]{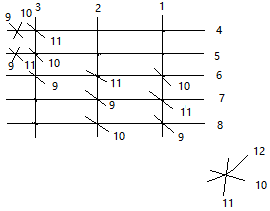}
\caption{}
\end{minipage}
\begin{minipage}[t]{0.5\linewidth}        %ֽ¼ֶ¬ױ¼׃ֳׂ»׀׀¿ם¶ָµִ45%
\hspace{2mm}
\includegraphics[width=6.5cm,height=4.5cm]{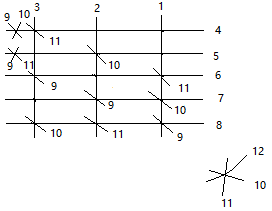}
\caption{}
\end{minipage}
\end{figure}

 After an easy computation, we conclude that Figures 25 and 26 cannot be realized.
\item[\rm 3)] $L_{11}$ passes through $\{L_1 \cap L_7, L_2 \cap L_6\}$. $L_{10}$ passes through $\{L_2\cap L_5, L_1 \cap L_6, L_3 \cap L_8\}$ (Figure 27).
\item[\rm 4)] $L_{11}$ passes through $\{L_1 \cap L_7, L_2 \cap L_8\}$. $L_{10}$ passes through $\{L_2\cap L_5, L_1 \cap L_6, L_3 \cap L_7\}$ (Figure 28).
    \begin{figure}[h]
\begin{minipage}[t]{0.5\linewidth}
\centering
\includegraphics[width=6.5cm,height=4.5cm]{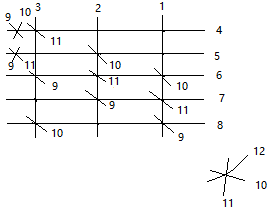}
\caption{}
\end{minipage}
\begin{minipage}[t]{0.5\linewidth}        %ֽ¼ֶ¬ױ¼׃ֳׂ»׀׀¿ם¶ָµִ45%
\hspace{2mm}
\includegraphics[width=6.5cm,height=4.5cm]{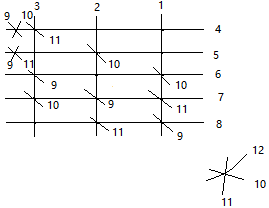}
\caption{}
\end{minipage}
\end{figure}

 After an easy computation, we conclude that Figures 27 and 28 cannot be realized.
\item[\rm 5)] $L_{11}$ passes through $\{L_1 \cap L_7, L_2 \cap L_8\}$. $L_{10}$ passes through $\{L_2\cap L_5, L_1 \cap L_6, L_3 \cap L_8\}$ (Figure 29).
\begin{figure}[h]
\centering
\includegraphics[width=6.5cm,height=4.5cm]{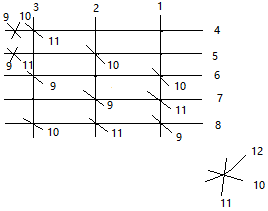}
\caption{}
\end{figure}
 After an easy computation, we conclude that Figures 29 cannot be realized.

    \end{enumerate}

\item[Subcase 4]$L_9$ pass through 3 points of $\{L_i\cap L_j:1\leq i\leq3,4\leq j\leq8\}$. $L_{10}$ pass through 3 points of $\{L_i\cap L_j:1\leq i\leq3,4\leq j\leq8\}$. $L_{11}$ pass through 2 points of $\{L_i\cap L_j:1\leq i\leq3,4\leq j\leq8\}$. Up to a lattice isomorphism, we assume $L_{9}$ passes through $\{L_3 \cap L_6, L_2 \cap L_7\}$, $L_{3}\cap L_{4}$ is on $L_{11}$. This subcase is lattice isomorphic to subcase 1.
\end{itemize}
\end{itemize}

\end{proof}
\section{Arrangements of $12$ lines with a sextic point and no quadruple point}

\subsection{One disjoint triple point apart from the pencil of the sextic point}
\begin{Theorem}
Let $\mathcal{A}$ be a nonreductive arrangement of $12$ lines in $\mathbb{CP}^2$ with a sexcit point $L_4 \cap L_5 \cap L_6 \cap L_7 \cap L_8 \cap L_{12}$ . Assume that $L_1 \cap L_2 \cap L_3 $ is the triple point apart from the sextic point; then there are exactly
$13$ triple points in $\mathcal{A}$. Then there is one case that can be realized, and this is ¡°potential Zariski pairs¡±.
\end{Theorem}

\begin{proof}
 Because $\{L_{9}\cup L{10}\cup L_{11}\}\cap \{L_i\cap L_j:4\leq i<j\leq8\}$ are at most $9$ points. Then  at least 3 multiple  points are not on
$\{L_i\cap L_j:1\leq i\leq3, 4\leq j\leq8\}$. But at most 3 multiple  points are not on
$\{L_i\cap L_j:1\leq i\leq3, 4\leq j\leq8\}$. Then there are exactly
$12$ triple points in $\mathcal{A}$.Up to a lattice isomorphism, we can assume the 3 multiple  points are $\{L_4\cap L_{10}\cap L_{11}, L_5\cap L_{9}\cap L_{11}, L_6\cap L_{10}\cap L_{9}\}$. And $L_{9}$ passes through $\{L_2 \cap L_7, L_3 \cap L_8\}$.
\begin{itemize}
\item[Case 1] $L_9$ pass through $L_1\cap L_4$. $L_{10}$ pass through $L_1\cap L_5$. $L_{11}$ pass through $L_1\cap L_6$.
    \begin{enumerate}
    \item[\rm 1)] $L_{11}$ passes through $\{L_2 \cap L_8, L_3 \cap L_{12}\}$. $L_{10}$ passes through $\{L_2\cap L_{12}, L_3 \cap L_{7}\}$ (Figure 30).
\item[\rm 2)] $L_{11}$ passes through $\{L_2\cap L_{12}, L_3 \cap L_{7}\}$. $L_{10}$ passes through $\{L_2 \cap L_8, L_3 \cap L_{12}\}$ (Figure 31).
    \begin{figure}[h]
\begin{minipage}[t]{0.5\linewidth}
\centering
\includegraphics[width=6.5cm,height=4.5cm]{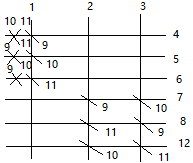}
\caption{}
\end{minipage}
\begin{minipage}[t]{0.5\linewidth}        %ֽ¼ֶ¬ױ¼׃ֳׂ»׀׀¿ם¶ָµִ45%
\hspace{2mm}
\includegraphics[width=6.5cm,height=4.5cm]{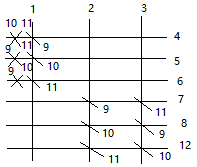}
\caption{}
\end{minipage}
\end{figure}
    \end{enumerate}

    After an easy computation, we conclude that Figures 30 and 31 cannot be realized.
    \item[Case 2] $L_9$ pass through $L_1\cap L_4$. $L_{10}$ pass through $L_1\cap L_5$. $L_{11}$ pass through $L_2\cap L_6$.
    \begin{enumerate}
    \item[\rm 1)] $L_{11}$ passes through $\{L_1 \cap L_8, L_3 \cap L_{12}\}$. $L_{10}$ passes through $\{L_2\cap L_{12}, L_3 \cap L_{7}\}$ (Figure 32).
\item[\rm 2)] $L_{11}$ passes through $\{L_1\cap L_{12}, L_3 \cap L_{7}\}$. $L_{10}$ passes through $\{L_2 \cap L_{8}, L_3 \cap L_{12}\}$ (Figure 33).
    \begin{figure}[h]
\begin{minipage}[t]{0.5\linewidth}
\centering
\includegraphics[width=6.5cm,height=4.5cm]{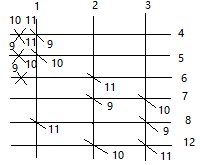}
\caption{}
\end{minipage}
\begin{minipage}[t]{0.5\linewidth}        %ֽ¼ֶ¬ױ¼׃ֳׂ»׀׀¿ם¶ָµִ45%
\hspace{2mm}
\includegraphics[width=6.5cm,height=4.5cm]{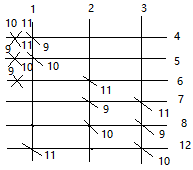}
\caption{}
\end{minipage}
\end{figure}
    \end{enumerate}

    After an easy computation, we conclude that Figures 32 and 33 cannot be realized.
  \item[Case 3] $L_9$ pass through $L_1\cap L_4$. $L_{10}$ pass through $L_2\cap L_5$. $L_{11}$ pass through $L_1\cap L_6$.
    \begin{enumerate}
    \item[\rm 1)] $L_{11}$ passes through $\{L_2 \cap L_8, L_3 \cap L_{12}\}$. $L_{10}$ passes through $\{L_1\cap L_{12}, L_3 \cap L_{7}\}$ (Figure 34).
\item[\rm 2)] $L_{11}$ passes through $\{L_2\cap L_{12}, L_3 \cap L_{7}\}$. $L_{10}$ passes through $\{L_1 \cap L_{8}, L_3 \cap L_{12}\}$ (Figure 35).
    \begin{figure}[h]
\begin{minipage}[t]{0.5\linewidth}
\centering
\includegraphics[width=6.5cm,height=4.5cm]{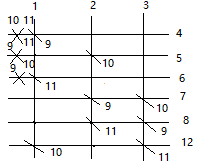}
\caption{}
\end{minipage}
\begin{minipage}[t]{0.5\linewidth}        %ֽ¼ֶ¬ױ¼׃ֳׂ»׀׀¿ם¶ָµִ45%
\hspace{2mm}
\includegraphics[width=6.5cm,height=4.5cm]{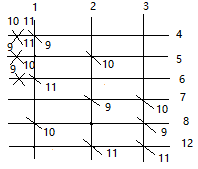}
\caption{}
\end{minipage}
\end{figure}
    \end{enumerate}

    After an easy computation, we conclude that Figures 34 and 35 cannot be realized.
    \item[Case 4] $L_9$ pass through $L_1\cap L_4$. $L_{10}$ pass through $L_2\cap L_5$. $L_{11}$ pass through $L_2\cap L_6$.
    \begin{enumerate}
    \item[\rm 1)] $L_{11}$ passes through $\{L_1 \cap L_8, L_3 \cap L_{12}\}$. $L_{10}$ passes through $\{L_1\cap L_{12}, L_3 \cap L_{7}\}$ (Figure 36).
\item[\rm 2)] $L_{11}$ passes through $\{L_1\cap L_{12}, L_3 \cap L_{7}\}$. $L_{10}$ passes through $\{L_1 \cap L_8, L_3 \cap L_{12}\}$ (Figure 37).
    \begin{figure}[h]
\begin{minipage}[t]{0.5\linewidth}
\centering
\includegraphics[width=6.5cm,height=4.5cm]{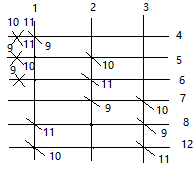}
\caption{}
\end{minipage}
\begin{minipage}[t]{0.5\linewidth}        %ֽ¼ֶ¬ױ¼׃ֳׂ»׀׀¿ם¶ָµִ45%
\hspace{2mm}
\includegraphics[width=6.5cm,height=4.5cm]{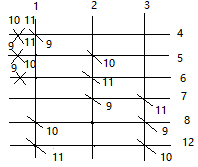}
\caption{}
\end{minipage}
\end{figure}
    \end{enumerate}

  After an easy computation, we conclude that Figures 36 and 37 cannot be realized.
   \item[Case 5] $L_9$ pass through $L_1\cap L_4$. $L_{10}$ pass through $L_2\cap L_5$. $L_{11}$ pass through $L_3\cap L_6$.
    \begin{enumerate}
    \item[\rm 1)] $L_{11}$ passes through $\{L_2 \cap L_8, L_1 \cap L_{12}\}$. $L_{10}$ passes through $\{L_1\cap L_{7}, L_3 \cap L_{12}\}$ (Figure 38).
\item[\rm 2)] $L_{11}$ passes through $\{L_1\cap L_{8}, L_2 \cap L_{12}\}$. $L_{10}$ passes through $\{L_1 \cap L_7, L_3 \cap L_{12}\}$ (Figure 39).
    \begin{figure}[h]
\begin{minipage}[t]{0.5\linewidth}
\centering
\includegraphics[width=6.5cm,height=4.5cm]{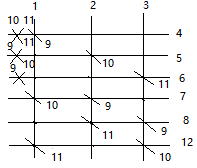}
\caption{}
\end{minipage}
\begin{minipage}[t]{0.5\linewidth}        %ֽ¼ֶ¬ױ¼׃ֳׂ»׀׀¿ם¶ָµִ45%
\hspace{2mm}
\includegraphics[width=6.5cm,height=4.5cm]{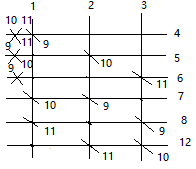}
\caption{}
\end{minipage}
\end{figure}

    \item[\rm 3)] $L_{11}$ passes through $\{L_1 \cap L_8, L_2 \cap L_{12}\}$. $L_{10}$ passes through $\{L_1\cap L_{12}, L_3 \cap L_{7}\}$ (Figure 40).
\item[\rm 4)] $L_{11}$ passes through $\{L_1\cap L_{7}, L_2 \cap L_{12}\}$. $L_{10}$ passes through $\{L_1 \cap L_8, L_3 \cap L_{12}\}$ (Figure 41).
    \begin{figure}[h]
\begin{minipage}[t]{0.5\linewidth}
\centering
\includegraphics[width=6.5cm,height=4.5cm]{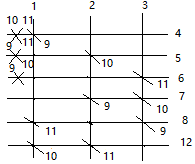}
\caption{}
\end{minipage}
\begin{minipage}[t]{0.5\linewidth}        %ֽ¼ֶ¬ױ¼׃ֳׂ»׀׀¿ם¶ָµִ45%
\hspace{2mm}
\includegraphics[width=6.5cm,height=4.5cm]{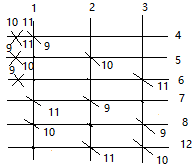}
\caption{}
\end{minipage}
\end{figure}
     \end{enumerate}

    After an easy computation, we conclude that Figures 38, 39 and 40 cannot be realized.
    Figure 41 can be defined by the following equation:
    $XY(Y-Z)(Y-t_1Z)(X-Z)(x-t_2Z)(x-t_3Z)(x-t_4Z)(X-t_5Z)(X+t_3Y-t_3Z)[X+(1-t_4)Y-Z][X+(t_5-t_3)Y-t_5Z]=0,$ where
    $t_1=1-\frac{1}{2t-t^2}, t_2=\frac{t-t^2+1}{2-t}, t_3=t, t_4=\frac{1}{2-t}, t_5=2t-t^2$, and satisfies $t^3-4t^2+3t+1=0$.
\end{itemize}
\end{proof}

\subsection{All triple points are in the pencil of the sextic point}
First, we show that there are at most 14 triple points in A.

\begin{Theorem}
 Let $\mathcal{A}$ be a nonreductive arrangement of $12$ lines with one sextic point so that all triple points
are on the lines passing through the sexcit point. Then there are at most 14 triple points and at least 12
triple points.

\end{Theorem}
\begin{proof}
From Lemma \ref{2} and Theorem \ref{3}, we have the following equation£÷
\begin{eqnarray*}
n_2 + 3n_3 +6n_4 + 10n_5 +15n_6 &=& 66\\
n_2 +\frac{3}{4}n_3 &\geq& 11 + n_5 +3n_6.
\end{eqnarray*}
From the above equations, we compute $n_3 \leq 16$.

 Let $L_1 \cap L_2 \cap L_3 \cap L_4 \cap L_5 \cap L_6$be the sextic point. each of the other
six lines passes through at least 3 and at most 6 triple points. Let a, b, c and d be the numbers of lines in
$\{L_7, L_8, L_9, L_10, L_{11}, L_{12}\}$ that pass through $3$, $4$, $5$, and $6$ triple points, respectively. If $d\neq0$, then one of $\{L_7, L_8, L_9, L_10, L_{11}, L_{12}\}$ pass through $6$ triple points. We can assume  $L_7$ pass through 6 triple points. But $L_7\cap (L_8\cup L_9\cup L_10\cup L_{11}\cup L_{12})\cap (L_1 \cup L_2 \cup L_3 \cup L_4 \cup L_5 \cup L_6)$ are at most $5$ points. Then $d=0$.

If $n_3 = 16$.
Then a, b, c, and d should satisfy
the following system of equations:

\begin{eqnarray*}
a+b+c+d &=& 6\\
3a+4b+5c+6d &=& 32.
\end{eqnarray*}
From the above equations, we have 4 solutions:
\begin{eqnarray*}
a&=&0, b = 0, c\ =\ 4, d\ =\ 2\\
a&=&0, b = 1, c\ =\ 2, d\ =\ 3\\
a&=&0, b = 2, c\ =\ 0, d\ =\ 2\\
a&=&1, b = 0, c\ =\ 1, d\ =\ 4
\end{eqnarray*}

 So $\mathcal{A}$ does not exist.

If $n_3 = 15$. Then a, b, c, and d should satisfy
the following system of equations:

\begin{eqnarray*}
a+b+c+d &=& 6\\
3a+4b+5c+6d &=& 30.
\end{eqnarray*}
From the above equations, we have $6$ solutions:
\begin{eqnarray*}
a&=&0, b\ =\ 1, c\ =\ 4, d\ =\ 1\\
a&=&0, b\ =\ 2, c\ =\ 2, d\ = 2\\
a&=&1, b\ =\ 0, c\ =\ 3, d\ =\ 2\\
a&=&0, b\ =\ 3, c\ =\ 0, d\ =\ 3\\
a&=&1, b\ =\ 1, c =\ 1, d\ =\ 3\\
a&=&2, b\ =\ 0, c\ =\ 0, d\ =\ 4
\end{eqnarray*}

 So $\mathcal{A}$ does not exist.

 Because each line of $\{L_1, L_2, L_3, L_4, L_5 L_6\}$ has at least two triple points, then $n_3 \geq 12$.

\end{proof}
\begin{Lemma}\label{cccc}
Let $\mathcal{A} = \{L_1, L_2,\cdots, L_n\}$ be a line arrangement. Assume that $L_n$ passes
through at least $2$ multiple points. Set $\mathcal{A}'= \{L_1, L_2, \cdots , L_{n-1}\}$, and then $\mathcal{M}_{\mathcal{A}}$ (or ${\mathcal{M}^c}_{\mathcal{A}}$ ) is a point or empty if $\mathcal{M}_{\mathcal{A}'}$ (or ${\mathcal{M}^c}_{\mathcal{A}'}$ ) is a point.
\end{Lemma}
\begin{proof}
We assume that  $\mathcal{M}_{\mathcal{A}}$  has two points  $\mathcal{A}$ and $\tilde{\mathcal{A}}.$ Then
$\mathcal{A}\sim \tilde{\mathcal{A}}/PGL_\mathbb{C}(2)$.
Let $\tilde{\mathcal{A}}= \{\tilde{L_1}, \tilde{L_2},\cdots, \tilde{L_n}\}$ and $\tilde{\mathcal{A}}' = \{\tilde{L_1}, \tilde{L_2},\cdots, \tilde{L_{n-1}}\}$ . There is a mapping $\varphi: L_i\rightarrow \tilde{L_i}$ such that $\varphi$ is an isomorphism. Then $\mathcal{A}'\sim \tilde{\mathcal{A}'}/PGL_\mathbb{C}(2)$. Because $\mathcal{M}_{\mathcal{A}'}$ is a point, then $\mathcal{A}'= \tilde{\mathcal{A}'}/PGL_\mathbb{C}(2)$. Because $L_n$ passes through at least $2$ multiple points and two points define a line, then $\mathcal{A}= \tilde{\mathcal{A}}/PGL_\mathbb{C}(2)$. Therefore $\mathcal{M}_{\mathcal{A}}$ is a point or empty. As the same proof, we can get ${\mathcal{M}^c}_{\mathcal{A}}$ is a point or empty if  ${\mathcal{M}^c}_{\mathcal{A}'}$  is a point.
\end{proof}
\begin{Theorem}
Let $\mathcal{A}$ be a nonreductive arrangement of $12$ lines with one sextic point and 14 triple points
such that all triple points are on the $6$ lines passing through the sexcit point. Then the quotient moduli space $\mathcal{M}^c_\mathcal{A}$ is
irreducible.
\end{Theorem}
\begin{proof}
Let $L_1 \cap L_2 \cap L_3 \cap L_4 \cap L_5 \cap L_6$be the sextic point. Since there are $14$ triple points on those $6$ lines and
we know that each of $6$ lines passes through at least $2$ and at most $3$ triple points, then we may assume that each
of $L_1$ and $L_2$  passes through 3 triple points. On the other hand, each of the other six lines passes through
at least $3$ and at most $5$ triple points. Let $a$, $b$, and $c$ be the numbers of lines in ${L_7, L_8, L_9, L_{10}, L_{11}, L_{12}}$
that pass through $3$, $4$, and $5$ triple points, respectively. Then a and b should satisfy the following system of
equations:
\begin{eqnarray*}
a + b + c &=& 6\\
3a + 4b + 5c &=& 28.
\end{eqnarray*}
From the above equations, we have two solutions:
\begin{eqnarray*}
a&=&1,\ b\ =\ 0,\ c\ =\ 5\\
a&=&0,\ b\ =\ 2,\ c\ =\ 4.
\end{eqnarray*}

Because there are $14$ triple points, one of ${L_1, L_2, L_3, L_4, L_5, L_6}$ has only two triple points. Sice $a\leq1$, we can let $L_6$ has only two triple points . Such that $\mathcal{A}'= \mathcal{A} \setminus L_6 $ is not reductive. Now $\mathcal{A}'$ is an arrangement of $11$ lines with $1$  quintuple point. All triple points are in the pencil of the quintuple point, and $\mathcal{A}'$ has $12$ triple points. From [\cite{Amr}, Theorem
6.6], $\mathcal{M}^c_{\mathcal{A}'}$
is one point.  Then the moduli space $\mathcal{M}^c_\mathcal{A}$ is one point or empty (by lemma \ref{cccc}).

\end{proof}
From the above discussions, we have the following corollary:

\begin{Corollary}
Let $\mathcal{A}$ be a nonreductive line arrangement of $12$ lines with $n_6 = 1$, $n_4 = 0$, and $n_r = 0$, $r \geq 7$.
Moreover, all triple points are on the $6$ lines passing through the sextic point. If it contains more than $13$
triple points, then there is no Zariski pair.
\end{Corollary}

Let $\mathcal{A}$ be a nonreductive arrangement of $12$ lines with a sextic point and all triple points are on the
$6$ lines passing through the sextic point. If the number of the triple points is less than $14$, then there are
many cases in which $\mathcal{M}_\mathcal{A}$ is more than $1$ point. Now we give 3 examples.
\begin{Example}
 The line arrangements are with $13$ triple points, and all triple points are on the $6$ lines passing
through the sextic point (see Figure 42).
\begin{figure}[h]
\centering
\includegraphics[width=6.5cm,height=4.5cm]{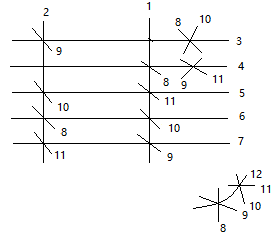}
\caption{}
\end{figure}
After some easy computation, we get the equation as follows: $XYZ(X-Z)(Y-Z)(X -t_1Z)(X-t_2Z)(X-t_3Z)[X+(t_2-1)Y-t_2Z][X+(t_2-t_1)Y-t_2Z](X-t_3Y)[X+(t_1-t_2)Y-t_1Z]= 0$, where $t_1=\frac{1}{2}, t_2=\pm\sqrt{\frac{1}{2}}, t_3=1\mp\sqrt{\frac{1}{2}}$.
 so that the moduli space $\mathcal{M}_\mathcal{A}$ is
two points.

\end{Example}
\begin{Example}
 The line arrangements are with $12$ triple points, and all triple points are on the $6$ lines passing
through the sextic point (see Figure 43).
\begin{figure}[h]
\centering
\includegraphics[width=6.5cm,height=4.5cm]{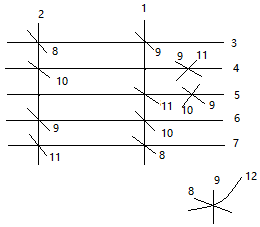}
\caption{}
\end{figure}
After some easy computation, we get the equation as follows: $XYZ(X-Z)(Y-Z)(X -t_1Z)(X-t_2Z)(X-t_3Z)(X+t_2Y-t_2Z)[X+(1-t_2)Y-Z](X-t_3Y)[X+(t_3-t_1)Y-t_3Z]= 0$, where $t_1=t^2-4t, t_2=t, t_3=-t,$ and $t$ satisfies $t^2-5t+2=0$.
 so that the moduli space $\mathcal{M}_\mathcal{A}$ is
two points.

\end{Example}
\begin{Example}
 The line arrangements are with $12$ triple points, and all triple points are on the $6$ lines passing
through the sextic point (see Figure 44).
\begin{figure}[h]
\centering
\includegraphics[width=6.5cm,height=4.5cm]{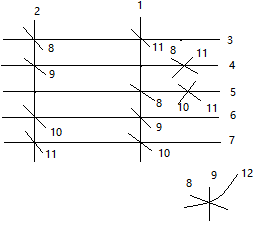}
\caption{}
\end{figure}
After some easy computation, we get the equation as follows: $XYZ(X-Z)(Y-Z)(X -t_1Z)(X-t_2Z)(X-t_3Z)[X+(1-t_2)Y-Z][X+(t_2-t_3)Y-t_2Z](X-t_1Y)[X+t_3Y-t_3Z]= 0$, where $t_1=t-1, t_2=t, t_3=\frac{t-1}{t-2},$ and $t$ satisfies $t^3-3t^2+2t-2=0$.
 so that the moduli space $\mathcal{M}_\mathcal{A}$ is
3 points.

\end{Example}

\end{document}